\documentclass[11pt,a4]{article}
\usepackage{latexsym}
\usepackage{amsmath}
\usepackage{amssymb}
\usepackage{amsthm}
\usepackage{amscd}
\usepackage{color}
\usepackage{citeref}
\usepackage{comment}
\definecolor{red}{rgb}{0.7,0,0}

\numberwithin{equation}{section}

\newtheorem{theorem}{Theorem}[section]
\newtheorem{lemma}[theorem]{Lemma}

\newtheorem{proposition}[theorem]{Proposition}

\theoremstyle{definition}
\newtheorem{definition}[theorem]{Definition}

\newtheorem{remark}[theorem]{Remark}

\theoremstyle{remark}

\newcommand{\supp}[1]{{\rm supp {#1}}}
\renewcommand{\eqref}[1]{(\ref{#1})}

\renewcommand{\bigskip}{\vspace{0.2cm}}

\begin{document}

\title{Young's inequality for Banach function spaces and its application to the maximal regularity estimate}


\author{Toru Nogayama \footnote{address:
Department of Mathematics,
Faculty of Science Division II,
Tokyo University of Science,
1-3 Kagurazaka, Shinjuku, Tokyo 162-8601, Japan
e-mail:toru.nogayama@gmail.com}
}

\date{}

\maketitle

{\bf keywords} Young's inequality, Banach function spaces, 
maximal regularity estimate

{\bf 2010 Classification}  42B25, 42B35



\begin{abstract}
The goal of this paper is to give the necessary and sufficient condition for Banach function spaces on which Young's inequality holds. 
As an application, we consider the maximal regularity estimate of heat equations for Besov spaces associated with Banach function spaces.
\end{abstract}

\section{Introduction}\label{intro}
In this paper, we consider 
Young's inequality for convolution operator on Banach function spaces satysfying some conditions.
Moreover, we apply it to obtain maximal regularity estimates for Besov spaces associated with Banach function spaces.
Here and below, we assume that 
$(X, \|\cdot\|_{X})$ is a Banach space which is cointained in $L^0({\mathbb R}^n)$,
the linear space of all measurable functions.
(We only consider the function on ${\mathbb R}^n$. 
So we omit ${\mathbb R}^n$ for $X$.)
We consider following conditions for X:
\begin{enumerate}
\item[(L)]
(Lattice property)
For all $f \in X$ and $g \in L^0({\mathbb R}^n)$,
if $|g| \le |f|$ holds, then $g \in X$ and the inequality $\|g\|_{X} \le \|f\|_{X}$ holds. 

\item[(F)]
(Fatou property)
If $0 \le f_n \uparrow f$ for $(f_n)_{n \in {\mathbb N}}$ in $X$
and $\sup\limits_{n \in {\mathbb N}} \|f_n\|_{X}<\infty$,
then $f \in X$ and 
$\|f\|_{X}=\sup\limits_{n \in {\mathbb N}} \|f_n\|_{X}$.

\item[(Si)]
For any measurable set $E \subset {\mathbb R}^n$ with finite measure, 
$\chi_{E} \in X$.

\item[(BSi)]
For any ball $B$ in ${\mathbb R}^n$,
$\chi_{B} \in X$.

\item[(LI)]
For any measurable set $E \subset {\mathbb R}^n$ with finite measure and $f \in X$, 
$\displaystyle \int_{E} |f| \le C\|f\|_{X}$,
where the constant $C$ is independent of $f$.

\item[(BLI)]
For any ball $B$ in ${\mathbb R}^n$ and $f \in X$, 
$\displaystyle \int_{B} |f| \le C\|f\|_{X}$,
where the constant $C$ is independent of $f$.

\item[(Sa)]
(Sarturation property)
For every measurable subset $E$ of ${\mathbb R}^n$ with positive measure, there exists a measurable set $F \subset E$ with nonzero measure such that $\chi_F \in X$.
\end{enumerate}

\begin{definition}[Banach function spaces]\label{def:250506-1}\,
\begin{enumerate}
\item
A Banach space $X$ is a Banach function space
if $X$ has the properties (L), (F), (Si) and (LI).

\item
A Banach space $X$ is a ball Banach function space 
if $X$ has the properties (L), (F), (BSi) and (BLI).

\item
A Banach space $X$ is a sarturated Banach function space if $X$ has the properties (L), (F) and (Sa).
\end{enumerate}
\end{definition}

Note that, by the definition, Banach function spaces are ball Banach function spaces.
Actually, we can see that ball Banach function spaces are sarturated Banach function spaces, 
see Lemma \ref{lem:250812-1}. 

Bennet and Sharpley (\cite{BS}) adopt the first definition in Definition \ref{def:250506-1}.
They are mainly forcused on the rearrangement-invariant Banach function spaces.
In this setting, the property (Sa) is equivalent to (Si) and (LI).
For more detail, we refer to \cite[Remark 2.4]{LN24}.
Lebesgue spaces, mixed Lebesgue spaces, Lorentz spaces, and Orlicz spaces are examples of Banach function spaces.
(See \cite{BS, SFH20-1}.)
However, the concept of Banach function spaces are a little restrictive.
For example, Morrey spaces are not Banach function spaces since the property (LI) fails.
To overcome this difficulty, Hakim and Sawano introduced the concept of ball Banach function spaces in \cite{HS16}.
See also \cite{SFH20-1}.
Morrey spaces and mixed Morrey spaces are example of ball Banach function spaces which are not Banach function spaces.
Meanwhile, sarturated Banach function spaces were appeared in \cite[Chapter 15]{Zaa}.
Lorist and Nieraeth gave the survey about sarturated Banch function spaces (\cite{LN24}).
In it, they consider sarturated (quasi) Banach function spaces without the Fatou property.
In this paper, we add the Fatou property since it is important to show the main theorem and applications.


The main theorem of this paper is following. 
\begin{theorem}\label{main1}
\begin{enumerate}
\item[\rm (1)]
Suppose that $X$ is 
a sarturated Banach function space.
If we have 
\begin{align} \label{eq:250421-1}
\|f(\cdot-z)\|_{X} \lesssim \|f\|_{X}
\end{align}
for all $f \in X$ and $z \in {\mathbb R}^n$, 
then Young's inequality
\[
\|f*g\|_{X} \lesssim \|f\|_{X}\|g\|_{L^1}
\] 
holds for all $f \in X$ and $g \in L^1$.

\item[\rm (2)]
Suppose that $X$ is a ball Banach function space.
Let $f \in X$. If Young's inequality
\[
\|f*g\|_{X} \lesssim \|f\|_{X}\|g\|_{L^1}
\] 
holds for all $g \in L^1$,
then we have the condition (\ref{eq:250421-1}) 
for all $z \in {\mathbb R}^n$. 

\end{enumerate}
\end{theorem}

We organize the remaining part of this paper as follows.
Section \ref{pre} is devoted to the preparation of some ingredients from harmonic analysis.
In Section \ref{proof}, we prove Theorem \ref{main1} and give some remarks.
As an application, we give the maximal regularity estimate on Besov spaces associated with Banach function spaces in Section \ref{app}.

\section{Preliminaries}\label{pre}

\subsection{Banach function spaces}

In this subsction, we summarize the relation and the properties 
to Banach function spaces
in Section \ref{intro}.
Recall that the Banach space $X$ is contained in $L^0({\mathbb R}^n)$.
First, we give the equivalent conditions to (Sa).

\begin{lemma}[{\cite[Proposition 2.5]{LN24}}] \label{lem:250106-1}
Let $X$ be a Banach space with (L).
Then, the followings are equivalent.

{\rm (i)} $X$ satisfies (Sa).

{\rm (ii)} There is an increasing sequence $F_n \subset {\mathbb R}^n$
with $\chi_{F_n} \in X$ and 
${\mathbb R}^n=\bigcup\limits_{n=1}^{\infty} F_n$.

{\rm (iii)} There is a function $u \in X$ such that $u>0$ a.e.

{\rm (iv)} If $g \in L^0({\mathbb R}^n)$ with 
$\displaystyle \int |f(x)g(x)| {\rm d}x=0$ for all $f \in X$,
then $g=0$ a.e.
\end{lemma}

By Lemma \ref{lem:250106-1}, we have the following relation between ball Banach function spaces and sarturated Banach function spaces.
(See also \cite[Remark 2.6]{Nie23}.)

\begin{lemma} \label{lem:250812-1}
The condition (BSi) implies the condition (Sa).
In particular, if $X$ is a ball Banach function space, then $X$ is a sarturted Banach function space.
\end{lemma}

Next, we recall the K\"othe dual.

\begin{definition}
Let $X$ be a Banach space with (L).
We define the K\"othe dual or associate space $X'$ of $X$
as the space 
\[
X'=\{g \in L^0({\mathbb R}^n): 
fg \in L^1({\mathbb R}^n) \mbox{ for all } f \in X\}.
\]
For $g \in X'$, we define the associate norm as
\[
\|g\|_{X'}=\sup\left\{\int_{{\mathbb R}^n} |f(x)g(x)| {\rm d}x
: f \in X, \|f\|_{X} \le 1 \right\}.
\]
\end{definition}

Actually, from this definition, we only show that $\|\cdot\|_{X'}$ is semi-norm (that is, there is a function $g \in X'$ satisfying $g \neq 0$ a.e. such that $\|g\|_{X'}=0$).
The next proposition suggests 
that the sarturation property is important for the duality.

\begin{proposition}[{\cite[Proposition 2.3]{Nie23}, \cite[Proposition 2.5]{LN24}}]
Let $X$ be a Banach space with (L).
Then, the semi-norm $\|\cdot\|_{X'}$ on $X'$ is norm  if and only if the property (Sa) holds for $X$.
\end{proposition}

The following is the relation to local integrability and K\"othe dual.

\begin{lemma}[{\cite[Remark 2.6]{Nie23}}] \label{lem:250702-1}
The properties (Si) and (LI) are eqivalent to (Si) for $X$ and $X'$, i.e. for all measurable set $E \subset {\mathbb R}^n$ with positive measure, $\chi_E \in X$ and $\chi_E \in X'$. 
\end{lemma}

Actually, the K\"othe dual of $X$ has the same property as $X$.

\begin{lemma}[{\cite[Theorem 3.2]{LN24}}]
Let $X$ be one of the Banach function spaces in Definition \ref{def:250506-1}.
Then, $X'$ is also the same type Banach function space.
\end{lemma}

The following lemma is very important to apply the duality arguments.
The proof  is in 
\cite[Theorem 2.7]{BS} and \cite[Theorem 71.1]{Zaa}.

\begin{theorem}[Lorentz--Luxemburg theorem]\label{thm:250506-3}
Let $X$ be 
the Banach space with (L) and (Sa).
Then, $X$ satisfies the Fatou property (F) if and only if $X''=X$ with norm coincidence.
\end{theorem}

Finally, if we have the boundedness of the Hardy--Littlewood maximal operator on $X$, the sarturation property leads to 
the properties (BLI) and (BSi).

\begin{lemma}[{\cite[Lemma 2.26]{Nie23}}] \label{lem:250702-2}
Let $X$ be a sarturated Banach function space.
Assume that the Hardy--Littlewood maximal operator $M$ is bounded on $X$, then $\chi_B \in X$ and $\chi_B \in X'$.
\end{lemma}

Note that combining Lemma \ref{lem:250702-1} and \ref{lem:250702-2}, 
we have the following lemma.

\begin{lemma}
Let $X$ be a sarturated Banach function space.
Assume that the Hardy--Littlewood maximal operator $M$ is bounded on $X$, then $X$ is a ball Banach function space.
In particular, if we have the boundedness of the Hardy--Littlewood maximal operator, then the concept of ball Banach function spaces is same as that of sarturated Banach function spaces.
\end{lemma}

To consider the Young inequality, we recall the Minkowski inequality.
Let $f(x,y)$ be a function such that for almost all fixed $y$,
$f \in X$ as a function of $x$,
the function $\|f(\cdot, y)\|_{X}$ is measurable,
and 
\[
\int \|f(\cdot, y)\|_{X} {\rm d}y<\infty.
\] 
Then, we have
\begin{align} \label{eq:250812-2}
\left\|
\int_{{\mathbb R}^n} f(\cdot, y)
{\rm d}y
\right\|_{X}
\le
\int_{{\mathbb R}^n}
\left\|
 f(\cdot, y)
\right\|_{X}
{\rm d}y.
\end{align}

Actually, we can generalize
the Minkowski inequality
in some sence: 
\[
\left\|
\int_{{\mathbb R}^n} f(\cdot, y)
{\rm d}y
\right\|_{X''}
\le
\int_{{\mathbb R}^n}
\left\|
 f(\cdot, y)
\right\|_{X}
{\rm d}y.
\]
This fact is mentioned in \cite[p.45-46]{KPS}.
For the completeness, we give the proof.
By the definition of K\"othe dual and Fubini's theorem, we have
\begin{align*}
\left\|
\int_{{\mathbb R}^n} f(\cdot, y)
{\rm d}y
\right\|_{X''}
&=
\sup_{g \in X', \|g\|_{X'}\le 1}
\int_{{\mathbb R}^n}
\left(\int_{{\mathbb R}^n} f(x,y) {\rm d}y\right) 
g(x) {\rm d}x\\
&=
\sup_{g \in X', \|g\|_{X'}\le 1}
\int_{{\mathbb R}^n}
\left(\int_{{\mathbb R}^n} f(x,y)g(x) {\rm d}x\right) 
 {\rm d}y.
\end{align*}
Using the H\"older inequality (\cite[Theorem 2.4]{BS}), we obtain 
\begin{align*}
\left\|
\int_{{\mathbb R}^n} f(\cdot, y)
{\rm d}y
\right\|_{X''}
\le
\sup_{g \in X', \|g\|_{X'}\le 1}
\int_{{\mathbb R}^n}
\|f(\cdot,y)\|_{X}\|g\|_{X'} 
{\rm d}y
\le
\int_{{\mathbb R}^n}
\|f(\cdot,y)\|_{X}
{\rm d}y.
\end{align*}

In particular, let $X$ be a one of the Banach function spaces in Definition \ref{def:250506-1}.
Then, by Thorem \ref{thm:250506-3}, $X''$ coincides to $X$ with norm coincidence.
Thus, the above generalized Minkowski inequality coincides to classical one.

\subsection{Tools from harmonic analysis}

In this subsection, we prepare some ingredients from harmonic analysis.

\subsubsection{Maximal operator}
First, we recall the Hardy--Littlewood maximal operator.
Define the Hardy--Littlewood maximal operator $M$ as
\[
M f(t)\equiv\sup_{a,b:0<a<t<b<\infty}
\frac{1}{b-a}\int_a^b |f(s)|{\rm d}s
\quad(0<t<\infty)
\]
for a measurable function $f$ on $(0,\infty)$.
In the same way, we also define the maximal operator for the function on ${\mathbb R}^n$, that is,
\[
Mf(x)\equiv
\sup_{B}\dfrac{\chi_B(x)}{|B|}
\int_{{\mathbb R}^n} |f(y)| {\rm d}y
\quad
(f \in L^0({\mathbb R}), x \in {\mathbb R}^n),
\]
where the supremum is taken over for all ball $B$ 
in ${\mathbb R}^n$. 
(We use the same notation $M$ if there is no confusion.)
The maximal operator $M$ is bounded on $L^p({\mathbb R}^n)$
for $1<p<\infty$.
Thanks to this fact, we obtain the Lebesgue differential theorem.

\begin{lemma}[{\cite[Theorem 1.48]{Sa18}}] \label{lem:250905-1}
Let $f \in L^1_{\rm loc}({\mathbb R}^n)$.
Then
\[
\lim_{r \to 0} \dfrac{1}{|B(x,r)|}\int_{B(x,r)} |f(y)-f(x)| {\rm d}y=0
\]
for a.e. $x \in {\mathbb R}^n$.
Here, $B(x,r)$ is a ball with center $x \in {\mathbb R}^n$ and radius $r>0$.
\end{lemma}

The remarkable property of maximal operator is 
that we can estimate the convolution operator
which kernel has appropriate decay.

\begin{lemma}[{\cite[Lemma 2.7]{NS24}}]\label{lem:221222-1}
Let $a>0$ and $t>0$.
For all non-negative
measurable functions $f=f(s)$ on $(0, \infty)$,
we have
\[
\int_0^t ae^{-a(t-s)} f(s) {\rm d}s
\le (1+e^{-1})Mf(t).
\]
\end{lemma}

The following is very important maximal inequality for sequences of functions.

\begin{lemma}[Fefferman--Stein vector-valued maximal inequality]\label{lem:250829-1}
Let $1<\rho<\infty$
and
$1<\sigma \le \infty$.
Then for all sequences
$\{f_j\}_{j=1}^\infty$ of measurable functions over $(0,\infty)$,
\[
\int_0^\infty 
\left(
\sum_{j=1}^\infty M f_j(t)^\sigma
\right)^{\frac{\rho}\sigma}{\rm d}t
\lesssim
\int_0^\infty 
\left(
\sum_{j=1}^\infty |f_j(t)|^\sigma
\right)^{\frac{\rho}\sigma}{\rm d}t.
\]
Here a natural modification is made if $\sigma=\infty$.
\end{lemma}

Remark that we have the Fefferman--Stein vector-valued maximal inequality over  ${\mathbb R}$.
See \cite{Sa18} for example.
A trivial zero extension and a restriction allow us to work in $(0,\infty)$.

\subsubsection{Interpolation}

We recall the real interpolation of $L^p$ spaces.
They will be used in Section \ref{app}.
First, we state the real interpolation result
for the vector-valued $L^p$ spaces.
Note that the results in this subsection hold 
for the Banach space $X$ which is not contained $L^0({\mathbb R}^n)$. 

\begin{lemma}[{\cite[5.3.1 Theorem]{BeLo76}}]\label{lem:250808-3}
Assume that $X$ is a Banach space.
Let $0<p_0<p_1 \le \infty$, $0<q \le \infty$ and $0<\theta<1$.
Then,
\[
\left(L^{p_0}(X), L^{p_1}(X)\right)_{\theta, q}
=L^{p,q}(X),
\]
where $1/p=(1-\theta)/p_0+\theta/p_1$.
\end{lemma}

To consider the real interpolation of Besov spaces, we introduce some notations.
Let $X$ be a Banach space 
and let $s \in {\mathbb R}$ and $q>0$.
Then, for $(a_m) \subset X$, we define
\[
\|(a_m)\|_{\dot{\ell}^s_q(X)}
=
\left(
\sum_{m \in {\mathbb Z}} (2^{sm} \|a_m\|_{X})^q\right)^{\frac1q}
\]
and $\dot{\ell}^s_q(X)$ is the set of all sequences 
$(a_m) \subset X$ such that $\|(a_m)\|_{\dot{\ell}^s_q(X)}<\infty$.

In particular, if we take $X=L^p$ and $a_m=\varphi_m(D)f$ 
for $f \in {\mathcal S}' / {\mathcal P}$,
then
$\dot{\ell}^s_q(X)$ stands for homogeneous Besov spaces.
Here, 
${\mathcal S}'$ is the set of all tempered distributions,
${\mathcal P}$ is the set of all polynomials, and
$\varphi$ is a suitable smooth function.
(See Section \ref{app} for details.)

We can consider the real interpolation space for $\dot{\ell}^s_q(X)$.
First, we recall the case $X=X_0=X_1$, which is a Banach space.

\begin{lemma}[{\cite[5.6.1 Theorem]{BeLo76}}] \label{lem:250808-1}
Let $0<\theta<1$.
Assume that $0<q_0 \le \infty$, $0<q_1 \le \infty$ 
and $s_0, s_1 \in {\mathbb R}$ with $s_0 \neq s_1$.
Then, for all $0<q \le \infty$, we have
\[
\left(
\dot{\ell}^{s_0}_{q_0}(X), \dot{\ell}^{s_1}_{q_1}(X)
\right)_{\theta, q}
=
\dot{\ell}^{s}_{q}(X),
\]
where $s=(1-\theta)s_0+\theta s_1$.

Moreover, if $s=s_0=s_1$, then we have
\[
\left(
\dot{\ell}^{s}_{q_0}(X), \dot{\ell}^{s}_{q_1}(X)
\right)_{\theta, q}
=
\dot{\ell}^{s}_{q}(X)
\]
provided that $1/q=(1-\theta)/q_0+\theta/q_1$.
\end{lemma}

Meanwhile, for the case $X_0 \neq X_1$,
the reader refers to \cite[5.6.2 Theorem]{BeLo76}.


\section{Proof of Theorem \ref{main1}} \label{proof}

First, we will show (1).
Suppose that (\ref{eq:250421-1}) holds for all $f \in X$ and $z \in {\mathbb R}^n$.
Let $g \in L^1$.
Then, by the Minkowski inequality (\ref{eq:250812-2}), we have
\begin{align}\label{eq:250505-1}
\|f*g\|_{X}
&=
\left\|\int_{{\mathbb R}^n} f(\cdot-y)g(y) {\rm d} y\right\|_{X}\nonumber
\\
&\le
\int_{{\mathbb R}^n} \left\|f(\cdot-y)\right\|_{X}|g(y)| {\rm d} y
\lesssim
\int_{{\mathbb R}^n} \left\|f\right\|_{X}|g(y)| {\rm d} y
=
\|f\|_{X}\|g\|_{L^1}.
\end{align} 
In third inequality, we use the condition (\ref{eq:250421-1}).
Thus, we obtain Young's inequality.

Meanwhile, let $f \in X$.
Suppose that Young's inequality holds for the function $f$.
Let $z \in {\mathbb R}^n$.
For $k \in {\mathbb N}$, 
let $g_{k}=k^n\chi_{[0,1]^n}(k(\cdot-z))$.
Clearly, $g_k \in L^1$ and $\|g_k\|_{L^1}=1$.
Moreover, for $x \in {\mathbb R}^n$ we have
\[
f*g_k(x)
=
\int_{{\mathbb R}^n} f(x-y) k^n\chi_{[0,1]^n}(k(y-z)) {\rm d}y
=
\int_{[0,k^{-1}]^n} k^n f(x-y-z) {\rm d}y.
\]
Hence, by the Young inequality, we obtain
\[
\left\|
\dfrac{1}{k^{-n}}\int_{[0,k^{-1}]^n} f(\cdot-y-z) {\rm d}y
\right\|_{X}
\lesssim
\|f\|_{X}\|g_k\|_{L^1}
=\|f\|_{X}.
\]
Finally, using Lemma \ref{lem:250905-1}
and the Fatou property, we have
\begin{equation} \label{eq:250506-4}
\|f(\cdot-z)\|_X
\le
\liminf_{k \to \infty}
\left\|
\dfrac{1}{k^{-n}}\int_{[0,k^{-1}]^n} f(\cdot-y-z) {\rm d}y
\right\|_{X}
\lesssim
\|f\|_{X}\|g_k\|_{L^1}
=\|f\|_{X}.
\end{equation}
Thus, we obtained the desired result.
(Remark that we can apply Lemma \ref{lem:250905-1}
since $X$ has the property (BLI).)

\begin{remark}
By (2) in Theorem \ref {main1}, if Young's inequality 
$\|h*g\|_{X} \lesssim \|h\|_{X}\|g\|_{L^1}$
holds for all $h \in X$ and $g \in L^1$, then
for all $f \in X$ and $z \in {\mathbb R}^n$,
the condition
\[
\|f(\cdot-z)\|_{X} \lesssim \|f\|_{X}
\]
holds.
Indeed, let $f \in X$ and $z \in {\mathbb R}^n$.
By assumption, $\|f*g\|_{X} \lesssim \|f\|_{X}\|g\|_{L^1}$
holds for all $g \in L^1$.
Thus, by (2) in Theorem \ref {main1}, 
$\|f(\cdot-z)\|_{X} \lesssim \|f\|_{X}$ holds.
Thus, if $X$ is a ball Banach function space, the condition (\ref{eq:250421-1}) is necessary and sufficient condition for Young's inequality on $X$.
\end{remark}

\begin{remark}
By Theorem \ref{thm:250506-3}, if we remove the Fatou property from the definition of Banach function spaces, the Young's inequality (\ref{eq:250505-1}) is replaced by
\[
\|f*g\|_{X''}
\lesssim
\|f\|_{X}\|g\|_{L^1}.
\]
In this case, the condition (\ref{eq:250421-1}) seems not to be necessary condition for Young's inequality 
since we cannot use the argument (\ref{eq:250506-4}).
\end{remark}

\section{Application to the maximal regularity estimate}\label{app}

In this section, we apply the Young inequality for the maximal regularity estimate for heat equations.
Consider the heat equation of the form:
\begin{equation}\label{eq:heat}
\begin{cases}
\partial_t u-\Delta u=f&\mbox{ in }(0, \infty) \times {\mathbb R}^n,\\
u(0,\cdot)=u_0&\mbox{ on } {\mathbb R}^n
\end{cases}
\end{equation}
We deal with the case where $f$ and $u_0$ belong to some Besov spaces. 
For that purpose, we define Besov spaces associated with Banach function spaces.
Here and below, we assume that $X$ is a sarturated Banch function spaces with the property (\ref{eq:250421-1}).

We adopt the following definition of the Fourier transform.
For $f \in L^1({\mathbb R}^n)$, define its Fourier transform and inverse Fourier transform by
\begin{eqnarray*}
{\mathcal F}f(\xi)
\equiv
(2\pi)^{-\frac{n}{2}}
\int_{{\mathbb R}^n} f(x){\rm e}^{-{\rm i} x \cdot \xi} {\rm d}x, \quad
{\mathcal F}^{-1}f(x)
\equiv
(2\pi)^{-\frac{n}{2}}
\int_{{\mathbb R}^n} f(\xi){\rm e}^{{\rm i} x \cdot \xi} {\rm d}\xi,
\end{eqnarray*}
respectively.
By a well-known method, we can extend 
${\mathcal F},{\mathcal F}^{-1}$ naturally to
${\mathcal S}'({\mathbb R}^n)$.
For $\psi \in {\mathcal S}({\mathbb R}^n)$ 
and $f \in {\mathcal S}'({\mathbb R}^n)$,
define $\psi(D)f
={\mathcal F}^{-1}\left[\psi {\mathcal F}f\right] \in C^{\infty}({\mathbb R}^n)$.

Denote by $B(r)$ the open ball centered at the origin of radius $r>0$.
Let $\varphi \in C^\infty_{\rm c}({\mathbb R}^n)$ satisfy
\[
\chi_{B(4) \setminus B(2)} \le \varphi \le 
\chi_{B(8) \setminus B(1)}.
\]
Then define
$\varphi_j\equiv \varphi(2^{-j}\cdot)$.
Denote by
${\mathcal P}({\mathbb R}^n) \subset {\mathcal S}'({\mathbb R}^n)$
the subspace of all polynomials.

\begin{definition}[Homogeneous Besov space associated with $X$]
Let $s \in {\mathbb R}$ and $1\le r \le \infty$.
We define
\[
\| f \|_{\dot{B}^s_{X, r}}
\equiv
\left(
\sum_{j=-\infty}^\infty 
\left(2^{j s}\left\| \mathcal{F}^{-1}\left[\varphi_j \mathcal{F}f\right] \right\|_{X}\right)^r
\right)^{\frac1r}
\]
for $f \in {\mathcal S}'({\mathbb R}^n)/{\mathcal P}({\mathbb R}^n)$.
The Besov spaces $\dot{B}^s_{X, r}({\mathbb R}^n)$ 
are defiend as the set of all 
$f \in {\mathcal S}'({\mathbb R}^n)/{\mathcal P}({\mathbb R}^n)$
such that
$\| f \|_{\dot{B}^s_{X, r}}$ is finite.
\end{definition}

This space is a special case of generalized Besov type space,
see \cite{LYYS13}. 
Note that by the embedding of $\ell^r$, we have
\begin{align}\label{eq:250812-5}
\dot{B}^s_{X, r_1}({\mathbb R}^n)
\hookrightarrow
\dot{B}^s_{X, r_2}({\mathbb R}^n)
\end{align}
for $1 \le r_1 \le r_2 \le \infty$ and $s \in {\mathbb R}$.

\begin{lemma}\label{lem:lift}{\rm \cite{LYYS13}}
Let $s \in {\mathbb R}$ and $1 \le r \le \infty$.
\begin{enumerate}
\item[$(1)$]
For all $k=1,2,\ldots,n$,
the $k$-th partial derivative
$\partial_k:\dot{B}^s_{X, r}({\mathbb R}^n) \to \dot{B}^{s-1}_{X, r}({\mathbb R}^n)$
is a continuous operator.
\item[$(2)$]
Let $\alpha \in {\mathbb R}$.
Then
$(-\Delta)^\alpha:\dot{B}^s_{X, r}({\mathbb R}^n) \to 
\dot{B}^{s-2}_{X, r}({\mathbb R}^n)$
is an isomorphism
with inverse $(-\Delta)^{-\alpha}:\dot{B}^{s-2\alpha}_{X, r}({\mathbb R}^n)
\to \dot{B}^s_{X, r}({\mathbb R}^n)$.
\item[$(3)$]
Let $\lambda \ge 0$.
Then
$(\lambda-\Delta)^{-1}:\dot{B}^s_{X, r}({\mathbb R}^n) \to \dot{B}^{s+2}_{X, r}({\mathbb R}^n)$
is a continuous operator.
\end{enumerate}
\end{lemma}
The operator $(-\Delta)^\alpha$ is called
the lift operator.

We state the maximal regularity estimates for $\dot{B}^s_{X, r}$.
The first one is $L^{\rho}$ maximal regularity estimate for heat equations (\ref{eq:heat}).

\begin{theorem}\label{main2}
Let $1 \le \rho \le \infty$. 
Consider the heat equation 
$(\ref{eq:heat})$
with $u_0 \in \dot{B}_{X, \rho}^{2-2/\rho}({\mathbb R}^n)$
and
$f \in L^\rho(0,\infty;\dot{B}^0_{X, \rho}({\mathbb R}^n))$.
Then
\begin{align*}
\|\partial_t u\|_{L^\rho(0,\infty;\dot{B}^0_{X, \rho})}
&+
\|\Delta u\|_{L^\rho(0,\infty;\dot{B}^0_{X, \rho})}
\lesssim
\|u_0\|_{\dot{B}_{X,\rho}^{2-2/\rho}}
+
\|f\|_{L^\rho(0,\infty;\dot{B}^0_{X, \rho})}.
\end{align*}
\end{theorem}

The second one is Lorentz maximal regularity estimate for heat equations (\ref{eq:heat}).
\begin{theorem}\label{main3}
Let $1 < \rho < \infty$, $1 \le w \le \infty$ and 
$1\le \sigma\le \infty$.
Consider the heat equation 
$(\ref{eq:heat})$
with $u_0 \in \dot{B}_{X, w}^{2-2/\rho}({\mathbb R}^n)$
and
$f \in L^{\rho, w}(0,\infty;\dot{B}^0_{X, \sigma}({\mathbb R}^n))$.
Then
\begin{align*}
\|\partial_t u\|_{L^{\rho, w}(0,\infty;\dot{B}^0_{X, \sigma})}
&+
\|\Delta u\|_{L^{\rho, w}(0,\infty;\dot{B}^0_{X, \sigma})}
\lesssim
\|u_0\|_{\dot{B}_{X, w}^{2-2/\rho}}
+
\|f\|_{L^{\rho. w}(0,\infty;\dot{B}^0_{X, \sigma})}.
\end{align*}
\end{theorem}
Note that by the lift operator
(Lemma \ref{lem:lift}),
we can genaralize the index of 
Besov space to $s \in {\mathbb R}$.

The maximal regularity estimate is one of the powerful tools for partial differential equations.
For example, combining these estimates and the fixed point argument, we can obtain the existence and uniqueness of solutions to quasi-linear partial differential equations.
Many researchers investigate this estimates in general setting.
(See \cite{DoVe87, KW04, LSU68, Weis01} and the references therein.)
However we can't apply these results for non-reflexive function spaces.
Hence, if we need the maximal regularity estimate for these spaces, we have to consider for each cases.
For example, Ogawa and Shimizu investigated the estimate for Besov space $\dot{B}^{s}_{1,r}$ in \cite{OgSh10}.
For Besov--Morrey spaces, the author and Sawano proved in
\cite{NS23, NS24}.
Theorems \ref{main2} and \ref{main3} include the non-reflexive case.
(For example, when $X$ is Morrey space, 
then $\dot{B}^s_{X, r}$ is homogeneous Besov--Morrey space, which is not reflexive.)

We move on to the proof.
By the Duhamel formula, we consider the corresponding integral equations
\[
u(t)=e^{t\Delta}u_0
-\int_0^t e^{(t-s)\Delta}f(s) {\rm d}s
\quad (t>0).
\]
First, we consider the linear term.

\begin{lemma}\label{thm:1.9}
Let $1 \le w \le \infty$ and $1 \le \rho<\infty$.
Then
\[
\|\Delta e^{t\Delta}u_0\|_{L^{\rho,w}(0,\infty;\dot{B}^{0}_{X, 1})}
\lesssim
\|u_0\|_{\dot{B}^{2-2/\rho}_{X, w}}
\]
for all $u_0 \in \dot{B}^{2-2/\rho}_{X, w}({\mathbb R}^n)$.
\end{lemma}

\begin{proof}
It suffices to show that 
\begin{align}\label{eq:250612-1}
\|\Delta e^{t\Delta}u_0\|_{L^{\tau}(0,\infty;\dot{B}^{0}_{X, 1})}
\lesssim
\|u_0\|_{\dot{B}^{2-2/\tau}_{X, \tau}}
\end{align}
for all $1 \le \tau \le \infty$.
Indeed, if we have (\ref{eq:250612-1}), 
by Lemmas \ref{lem:250808-3} and \ref{lem:250808-1}, for 
$\dfrac1\rho=\dfrac{1}{2\rho_1}+\dfrac{1}{2\rho_2}$ 
$(1\le \rho_1<\rho< \rho_2 \le \infty)$,
\[
L^{\rho,w}(0,\infty;\dot{B}^{0}_{X, 1})
=
\left(L^{\rho_1}(0,\infty;\dot{B}^{0}_{X, 1}), 
L^{\rho_2}(0,\infty;\dot{B}^{0}_{X, 1})\right)_{1/2, w}
\]
and
\[
\dot{B}^{2-2/\rho}_{X, w}
=
\left(
\dot{B}^{2-2/\rho_1}_{X, \rho_1},
\dot{B}^{2-2/\rho_2}_{X, \rho_2}
\right)_{1/2, {w}}
\]
for all $w \in [1, \infty]$.

Thus, we concentrate on (\ref{eq:250612-1}).
Let $\tau=\infty$. Then, by Young's inequality, we have
\begin{align*}
\lefteqn{
\|\Delta e^{t\Delta}u_0\|_{\dot{B}_{X, 1}^0}
}\\
&=
\sum_{j=-\infty}^\infty
\|\Delta\varphi_j(D) e^{t\Delta}u_0\|_{X}\\
&\lesssim
\sum_{j=-\infty}^\infty
\left\|{\mathcal F}^{-1}[e^{-t|\cdot|}]\right\|_{L^1}
\|\Delta\varphi_j(D) u_0\|_{X}
\sim
\sum_{j=-\infty}^\infty
\|\Delta\varphi_j(D) u_0\|_{X}
=
\|u_0\|_{\dot{B}_{X, 1}^2}.
\end{align*}

Let $1<\tau<\infty$.
We write the left-hand side out in full:
\begin{align*}
\left(
\int_0^\infty
(\|\Delta e^{t\Delta}u_0\|_{\dot{B}_{X, 1}^0})^\tau{\rm d}t
\right)^{\frac{1}{\tau}}
&=\left\{
\int_0^\infty
\left(
\sum_{j=-\infty}^\infty
\|\Delta\varphi_j(D) e^{t\Delta}u_0\|_{X}
\right)^{\tau}{\rm d}t
\right\}^{\frac{1}{\tau}}.
\end{align*}
Let $\Phi \in C^\infty_{\rm c}({\mathbb R}^n)$ be 
a radial function
that vanishes on $B(1)$
and assumes $1$ on the support of $\varphi$.
We write
$\Phi_j(\xi)\equiv\Phi(2^{-j}\xi)$.
Then by the Fourier transform, we obtain
\begin{align*}
\Delta \varphi_j(D)F
&=\Delta \Phi_j(D) \varphi_j(D)F
\simeq
4^{j} \left(
 2^{jn}{\mathcal F}^{-1}[|\cdot|^2\Phi](2^j\cdot)
\right)
*
\varphi_j(D)F
\end{align*}
By the Young's inequality
$\|\Psi*G\|_{X} \le \|\Psi\|_{L^1}\|G\|_{X}$
and 
the fact
\[
\left\|
2^{jn}{\mathcal F}^{-1}[|\cdot|^2\Phi](2^j\cdot)
\right\|_{L^1}
=
\left\|
{\mathcal F}^{-1}[|\cdot|^2\Phi]
\right\|_{L^1}
<\infty,
\]
we obtain
\begin{align*}
\left(
\int_0^\infty
(\|\Delta e^{t\Delta}u_0\|_{\dot{B}_{X,1}^0})^\tau{\rm d}t
\right)^{\frac{1}{\tau}}
&\lesssim
\left\{
\int_0^\infty
\left(
\sum_{j=-\infty}^\infty 4^{j}
\|e^{t\Delta} \varphi_j(D) u_0\|_{X}
\right)^{\tau}{\rm d}t
\right\}^{\frac{1}{\tau}}\\
&=
\left\{
\int_0^\infty
\left(
\sum_{j=-\infty}^\infty 4^{j}
\|\Phi_j(D)e^{t\Delta} \varphi_j(D) u_0\|_{X}
\right)^{\tau}{\rm d}t
\right\}^{\frac{1}{\tau}}.
\end{align*}
Since $\supp \, \Phi_j \subset (|\xi| \ge 2^j)$, we have
$\|{\mathcal F}^{-1}[\Phi_je^{-t|\cdot|^2}]\|_{L^1}
\lesssim
e^{-4^{j}t}$.
Thus once again
by the Young inequality, we obtain
\begin{align*}
\left(
\int_0^\infty
(\|\nabla \exp(t\Delta)u_0\|_{\dot{B}_{X, 1}^0})^\tau{\rm d}t
\right)^{\frac{1}{\tau}}
\lesssim
\left\{
\int_0^\infty
\left(
\sum_{j=-\infty}^\infty 4^{j}\exp(-4^jt)
\| \varphi_j(D) u_0\|_{X}
\right)^{\tau}{\rm d}t
\right\}^{\frac{1}{\tau}}.\\
\end{align*}

We estimate the sum on the right-hand side by H\"older's inequality.
To this end, we let $\alpha,\beta>0$ satisfy 
$\alpha+\beta=1$ and $\beta<2/\tau$.
Then
by H\"{o}lder's inequality,
\begin{align*}
\lefteqn{
\sum_{j=-\infty}^\infty
4^j\exp(-4^jt)\|\varphi_j(D)u_0\|_{X}
}\\
&\le
\left\{
\sum_{j=-\infty}^\infty
(4^{j\alpha}\exp(-4^{j-1}t)
\|\varphi_j(D)u_0\|_{X})^{\tau}
\right\}^{\frac{1}{\tau}}
\left\{
\sum_{j=-\infty}^\infty
\left(
4^{j\beta}\exp(-3\cdot 4^{j-1}t)\right)^{\tau'}
\right\}^{\frac1{\tau'}}\\
&\lesssim
t^{-\beta}
\left\{
\sum_{j=-\infty}^\infty
(4^{j\alpha}\exp(-4^{j-1}t)
\|\varphi_j(D)u_0\|_{X})^{\tau}
\right\}^{\frac{1}{\tau}}.
\end{align*}
By taking the $\tau$-th power, we have
\begin{align*}
\left(
\sum_{j=-\infty}^\infty
4^j\exp(-4^jt)\|\varphi_j(D)u_0\|_{X}
\right)^{\tau}
&\lesssim
t^{-\beta\tau}
\sum_{j=-\infty}^\infty
(4^{j\alpha}\exp(-4^{j-1}t)
\|\varphi_j(D)u_0\|_{X})^{\tau}.
\end{align*}
If we integrate against $t>0$,
we obtain
\[
\left\{
\int_0^\infty
(\|\Delta \exp(t\Delta)u_0\|_{\dot{B}_{X, 1}^0})^\tau{\rm d}t
\right\}^{\frac{1}{\tau}}
\lesssim
\left\{
\sum_{j=-\infty}^\infty
(2^{j(2-\frac{2}{\tau})}\|\varphi_j(D)u_0\|_{X})^{\tau}
\right\}^{\frac{1}{\tau}},
\]
as required.
Finally, when $\tau=1$, 
same arguments as the case $1<\tau<\infty$ are valid.
\end{proof}

Next, we consider the nonlinear term.

\begin{lemma}\label{lem:1}
Let $1 \le \sigma \le \infty$ and $1 <\rho<\infty$.
Then, 
\begin{equation}\label{eq:220924-21}
\left\|\Delta\int_0^t  e^{(t-\tilde{t})\Delta}[f(\tilde{t})]{\rm d}\tilde{t}
\right\|_{L^\rho(0,\infty;\dot{B}^0_{X, \sigma})}
\lesssim
\|f\|_{L^\rho(0,\infty;\dot{B}^0_{X, \sigma})}
\end{equation}
for all $f \in L^\rho(0,\infty;\dot{B}^0_{X, \sigma}({\mathbb R}^n))$.
\end{lemma}

\begin{proof}
Let $1<\sigma<\infty$.
We write out the left-hand side in full: 
\begin{align}\label{eq:250505-4}
\lefteqn{
\left\|\Delta\int_0^t  e^{(t-\tilde{t})\Delta}[f(\tilde{t})]{\rm d}\tilde{t}
\right\|_{L^\rho(0,\infty;\dot{B}^0_{X, \sigma})}
}\nonumber\\
&=
\left\{
\int_0^\infty
\left\{\left\|
\int_0^t \Delta e^{(t-\tilde{t})\Delta}[f(\tilde{t})]{\rm d}\tilde{t}
\right\|_{\dot{B}^0_{X, \sigma}}
\right\}^{\rho}{\rm d}t
\right\}^{\frac{1}{\rho}}\nonumber\\
&=
\left\{
\int_0^\infty
\left[
\sum_{j=-\infty}^{\infty}
\left\{
\left\|
\int_0^t\Delta e^{(t-\tilde{t})\Delta}\varphi_j(D)f(\tilde{t})
{\rm d}\tilde{t}\right\|_{X}
\right\}^{\sigma}\right]^{\frac{\rho}{\sigma}}{\rm d}t
\right\}^{\frac{1}{\rho}}.
\end{align}

Since ${\rm supp}\,\varphi_j \subset B(2^{j+3}) \setminus B(2^{j-1})$, 
by
triangle inequality and Theorem \ref{main1}, we obtain
\begin{align*}
\left\|
\int_0^t\Delta e^{(t-\tilde{t})\Delta}\varphi_j(D)f(\tilde{t})
{\rm d}\tilde{t}\right\|_{X}
&\sim
\left\|
\int_0^t\Delta e^{(t-\tilde{t})\Delta}\left(\sum_{k=j-2}^{j+2}\varphi_k(D)\right)
\varphi_j(D)f(\tilde{t})
{\rm d}\tilde{t}\right\|_{X}\\
&\le
\int_0^t\sum_{k=j-2}^{j+2}
\left\|\Delta e^{(t-\tilde{t})\Delta}\varphi_k(D)
\varphi_j(D)f(\tilde{t})
\right\|_{X}{\rm d}\tilde{t}\\
&\lesssim
\int_0^t
\sum_{k=j-2}^{j+2}\left\| 
{\mathcal F}^{-1}\left[|\cdot|^2e^{-(t-\tilde{t})|\cdot|^2}\varphi_k\right]
\right\|_{L^1}
\left\|\varphi_j(D)f(\tilde{t})
\right\|_{X}{\rm d}\tilde{t}\\
&\sim
\int_0^t4^j e^{-4^j(t-\tilde{t})}
\left\|\varphi_j(D)f(\tilde{t})\right\|_{X}
{\rm d}\tilde{t}.
\end{align*}

Inserting this estimate to (\ref{eq:250505-4}), we have
\begin{align}\label{eq:2}
\nonumber
\lefteqn{
\left\|\Delta\int_0^t  e^{(t-\tilde{t})\Delta}[f(\tilde{t})]{\rm d}\tilde{t}
\right\|_{L^\rho(0,\infty;\dot{B}^0_{X, \sigma})}
}\\
&\lesssim
\left\{
\int_0^\infty
\left[
\sum_{j=-\infty}^{\infty}
\left\{
\int_0^t4^j e^{-4^j(t-\tilde{t})}
\left\|\varphi_j(D)f(\tilde{t})\right\|_{X}
{\rm d}\tilde{t}
\right\}^{\sigma}\right]^{\frac{\rho}{\sigma}}{\rm d}t
\right\}^{\frac{1}{\rho}}. 
\end{align}

By Lemma \ref{lem:221222-1},
we obtain
\begin{align*}
\left\|\Delta\int_0^t  e^{(t-\tilde{t})\Delta}[f(\tilde{t})]{\rm d}\tilde{t}
\right\|_{L^\rho(0,\infty;\dot{B}^0_{X, \sigma})}
\lesssim
\left\{
\int_0^\infty\left(
\sum_{j=-\infty}^{\infty}
M\left[\left\|\varphi_j(D)f(\cdot)\right\|_{X}\right](t)^{\sigma}
\right)^{\frac{\rho}{\sigma}}{\rm d}t
\right\}^{\frac{1}{\rho}}.
\end{align*}
Since $1<\rho, \sigma<\infty$,
we are in the position of using the Fefferman--Stein vector-valued maximal inequality
over $(0,\infty)$
to have
\begin{align*}
\left\|\Delta\int_0^t  e^{(t-\tilde{t})\Delta}[f(\tilde{t})]{\rm d}\tilde{t}
\right\|_{L^\rho(0,\infty;\dot{B}^0_{X, \sigma})}
&\lesssim
\left\|
\left(
\sum_{j=-\infty}^{\infty}
\left\|\varphi_j(D)f\right\|_{X}^{\sigma}
\right)^{\frac{1}{\sigma}}
\right\|_{L^{\rho}_t}
=
\|f\|_{L^\rho(0,\infty;\dot{B}^0_{X, \sigma})}.
\end{align*}
When $\sigma=\infty$, the same argument is valid.

Let $\sigma=1$.
By the duality argument, we only to show that
\begin{align}\label{eq-1}
\left\|\sup_{j \in {\mathbb Z}}\int_s^\infty  4^je^{-4^j(t-s)}[g_j(t)]{\rm d}t
\right\|_{L^{\rho'}_s(0,\infty)}
\lesssim
\left\|\sup_{j \in {\mathbb Z}}g_j\right\|_{L^{\rho'}_s(0,\infty)}
\end{align}
for all sequences $\{g_j\} \subset L^{\rho'}(0,\infty;\ell^\infty)$ of non-negative functions.
Indeed, if we have (\ref{eq-1}),
by (\ref{eq:2}) the estimate 
\begin{align*}
\left\|\Delta\int_0^t  e^{(t-s)\Delta}[f(s)]{\rm d}s
\right\|_{L^\rho(0,\infty;\dot{B}^0_{X, 1})}
\lesssim
\left\|
\sum_{j=-\infty}^{\infty}
\int_0^t4^j e^{-4^j(t-s)}
\left\|\varphi_j(D)f(s)\right\|_{X}
{\rm d}s
\right\|_{L^\rho(0,\infty)}
\end{align*}
is derived.
For $g_j \in L^{\rho'}(0,\infty;\ell^\infty)$ with 
$\|g_j\|_{L^{\rho'}(0,\infty;\ell^\infty)}=1$, 
using the duality argument and the Fubini theorem, we have
\begin{align*}
\lefteqn{
\left\|\Delta\int_0^t  e^{(t-s)\Delta}[f(s)]{\rm d}s
\right\|_{L^\rho(0,\infty;\dot{B}^0_{X, 1})}
}\\
&\lesssim
\int_0^\infty
\sum_{j=-\infty}^{\infty}\left[
\int_0^t4^j e^{-4^j(t-s)}
\left\|\varphi_j(D)f(s)\right\|_{X}
{\rm d}s
\right]
g_j(t)
{\rm d}t  \\
&=
\int_0^t
\int_s^\infty
\sum_{j=-\infty}^{\infty}
4^j e^{-4^j(t-s)}
\left\|\varphi_j(D)f(s)\right\|_{X}g_j(t)
{\rm d}t
{\rm d}s\\
&=
\int_0^t
\sum_{j=-\infty}^{\infty}
\left\|\varphi_j(D)f(s)\right\|_{X}
\left(
\int_s^\infty
4^j e^{-4^j(t-s)}
g_j(t)
{\rm d}t
\right)
{\rm d}s.
\end{align*}
Hence, by H\"older's inequality, we have

\begin{align*}
\lefteqn{
\left\|\Delta\int_0^t  e^{(t-s)\Delta}[f(s)]{\rm d}s
\right\|_{L^\rho(0,\infty;\dot{B}^0_{X, 1})}
}\\
&\lesssim
\left\|
\sum_{j=-\infty}^{\infty}
\left\|\varphi_j(D)f(s)\right\|_{X}
\right\|_{L^\rho}
\left\|
\sup_{j}
\int_s^\infty
4^j e^{-4^j(t-s)}
g_j(t)
{\rm d}t
\right\|_{L^{\rho'}}
\lesssim
\|f\|_{L^\rho(0,\infty;\dot{B}^0_{X, 1})}.
\end{align*}

So, we turn to show (\ref{eq-1}).
Fixed $s \in (0, \infty)$.
Let $\displaystyle G_j(t)=\int_s^t g_j(\tilde{t}) {\rm d}\tilde{t}$
for $s<t<\infty$.
Note that
$G_j(t) \le (t-s) M[g_j](s)$
and
$G_j'(t)=g_j(t)$ holds.
Hence, by integration by parts, we obtain
\begin{align*}
\int_s^\infty  4^je^{-4^j(t-s)}[g_j(t)]{\rm d}t
&=
\int_s^\infty  4^je^{-4^j(t-s)}[G_j'(t)]{\rm d}t\\
&=
\int_s^\infty  16^je^{-4^j(t-s)}[G_j(t)]{\rm d}t\\
&\le
M[g_j(s)]
\int_s^\infty  16^j(t-s)e^{-4^j(t-s)}{\rm d}t
=
M[g_j(s)].
\end{align*}

Using the Fefferman-Stein vector-valued inequality
(Lemma \ref{lem:250829-1}), we have the desired result.
\end{proof}

Finally, we consider the case $\rho=1$.

\begin{lemma}
For all 
$f \in L^1(0,\infty;\dot{B}^0_{X, 1}({\mathbb R}^n))$,
we have
\begin{equation}\label{eq:220924-21}
\left\|\Delta\int_0^t  e^{(t-\tilde{t})\Delta}[f(\tilde{t})]{\rm d}\tilde{t}
\right\|_{L^1(0,\infty;\dot{B}^0_{X, 1})}
\lesssim
\|f\|_{L^1(0,\infty;\dot{B}^0_{X, 1})}.
\end{equation}
\end{lemma}

\begin{proof}
In (\ref{eq:2}), letting $\rho=\sigma=1$, we have
\begin{align*}
\left\|\Delta\int_0^t  e^{(t-\tilde{t})\Delta}[f(\tilde{t})]{\rm d}\tilde{t}
\right\|_{L^1(0,\infty;\dot{B}^0_{X, 1})}
\lesssim
\int_0^\infty
\left[
\sum_{j=-\infty}^{\infty}
\int_0^t4^j e^{-4^j(t-\tilde{t})}
\left\|\varphi_j(D)f(\tilde{t})\right\|_{X}
{\rm d}\tilde{t}
\right]{\rm d}t. 
\end{align*}
By Fubini's theorem, we have
\begin{align*}
\left\|\Delta\int_0^t  e^{(t-\tilde{t})\Delta}[f(\tilde{t})]{\rm d}\tilde{t}
\right\|_{L^1(0,\infty;\dot{B}^0_{X, 1})}
&\lesssim
\int_0^\infty
\sum_{j=-\infty}^{\infty}
\left[
\int_{\tilde{t}}^{\infty}
4^j e^{-4^j(t-\tilde{t})}
{\rm d}\tilde{t}
\right]
\left\|\varphi_j(D)f(\tilde{t})\right\|_{X}
{\rm d}t\\
&\sim
\int_0^\infty
\sum_{j=-\infty}^{\infty}
\left\|\varphi_j(D)f(\tilde{t})\right\|_{X}
{\rm d}t\\
&=
\|f\|_{L^1(0,\infty;\dot{B}^0_{X, 1})}.
\end{align*}
We obtain the desired result.
\end{proof}
Let $1<\rho<\infty$ and $1 \le \sigma \le \infty$.
Thanks to Lemmas \ref{lem:250808-3}, 
for 
$\dfrac1\rho=\dfrac{1}{2\rho_1}+\dfrac{1}{2\rho_2}$ 
$(1\le \rho_1<\rho< \rho_2 \le \infty)$,
we obtain
\[
L^{\rho,w}(0,\infty;\dot{B}^{0}_{X, \sigma})
=
\left(L^{\rho_1}(0,\infty;\dot{B}^{0}_{X, \sigma}), 
L^{\rho_2}(0,\infty;\dot{B}^{0}_{X, \sigma})\right)_{1/2, w}
\]
for all $w \in [1, \infty]$.
Thus, the estimate
\begin{equation}\label{eq:250812-4}
\left\|\Delta\int_0^t  e^{(t-\tilde{t})\Delta}[f(\tilde{t})]{\rm d}\tilde{t}
\right\|_{L^{\rho,w}(0,\infty;\dot{B}^0_{X, \sigma})}
\lesssim
\|f\|_{L^{\rho, w}(0,\infty;\dot{B}^0_{X, \sigma})}
\end{equation}
holds 
for all $f \in L^{\rho,w}(0,\infty;\dot{B}^0_{X, \sigma}({\mathbb R}^n))$.

\begin{proof}[Proof of Theorems \ref{main2} and \ref{main3}]

Combining the estimate (\ref{eq:250612-1}),
the embedding (\ref{eq:250812-5}), 
and Lemma \ref{lem:1}, we have Theorem \ref{main2}.
Theorem \ref{main3} is led from Lemma \ref{thm:1.9} and
the estimate (\ref{eq:250812-4}).
\end{proof}

\section*{Acknowledgement}
The author was supported financially by Research Fellowships of the Japan Society for the Promotion of Science for Young Scientists (22J00614, 24K22839).

\end{document}